\documentclass[10pt,leqno]{amsart}
\usepackage{amssymb,amsthm}
\usepackage{amsmath}
\usepackage{setspace}
\usepackage{dcpic,pictexwd}
\usepackage[all]{xy}
\usepackage[pdftex]{graphicx}
\usepackage{mathrsfs}
\usepackage[pdftex]{hyperref}
\usepackage{bbm}

\theoremstyle{definition}
 \newtheorem{definition}{Definition}[section]

\theoremstyle{plain}

\theoremstyle{plain}
 \newtheorem{theorem}[definition]{Theorem}
 
\theoremstyle{definition}
 \newtheorem{example}[definition]{Example}

\theoremstyle{plain}
 \newtheorem{lemma}[definition]{Lemma}

\theoremstyle{plain}
 \newtheorem{corollary}[definition]{Corollary}

\theoremstyle{remark}
 \newtheorem{remark}[definition]{Remark}

\theoremstyle{definition}

\theoremstyle{plain}

\newcommand{\Ext}{\mathrm{Ext}}
\newcommand{\End}{\mathrm{End}}

\newcommand{\Hom}{\mathrm{Hom}}
\newcommand{\SHom}{\underline{\mathrm{Hom}}}

\newcommand{\Ca}{\mathcal{C}}
\newcommand{\Fun}{\mathrm{F}}

\newcommand{\Def}{\mathrm{Def}}
\newcommand{\Sets}{\mathrm{Sets}}

\newcommand{\Ob}{\mathrm{Ob}}

\newcommand{\Z}{\mathbb{Z}}
\newcommand{\SEnd}{\underline{\End}}
\newcommand{\A}{\Lambda}

\newcommand{\m}{\mathfrak{m}}
\renewcommand{\k}{\Bbbk}

\newcommand{\invlim}{\varprojlim}

\hypersetup{
    bookmarks=true,         
    unicode=false,          
    pdftoolbar=true,        
    pdfmenubar=true,        
    pdffitwindow=false,     
    pdfstartview={FitH},    
    pdftitle={On Singular Equivalences of Morita Type with Level and Universal Deformation Rings of Modules over Gorenstein Algebras},    
    pdfauthor={Jose A. Velez-Marulanda},     
    pdfsubject={},   
    pdfcreator={Jose A. Velez-Marulanda},   
    pdfproducer={Jose A. Velez-Marulanda}, 
    pdfkeywords={Universal deformation rings} {Singular Equivalences of Morita Type with Level}{Finitely generated Gorenstein-projective modules}, 
    pdfnewwindow=true,      
    colorlinks=true,       
    linkcolor=red,          
    citecolor=blue,        
    filecolor=magenta,      
    urlcolor=blue           
}

\title[Singular equivalences of Morita type with level and universal deformation rings]{Singular equivalences of Morita type with level, Gorenstein algebras, and universal  deformation rings} 
\dedicatory{To the memory and legacy of my great-grandfather Jes\'us Mar\'{\i}a Berm\'udez.}  
\thanks{The author was supported by the Faculty Research Scholarship of the College of Sciences and Mathematics at the Valdosta State University.}

\author{Jos\'e A. V\'elez-Marulanda}
\address{Department of Mathematics, Valdosta State University, Valdosta, GA, U.S.A.}
\email{javelezmarulanda@valdosta.edu}

\begin{document}
\renewcommand{\labelenumi}{\textup{(\roman{enumi})}}
\renewcommand{\labelenumii}{\textup{(\roman{enumi}.\alph{enumii})}}
\numberwithin{equation}{section}


\begin{abstract}
Let $\k$ be a field of arbitrary characteristic, let $\A$ be a finite dimensional $\k$-algebra, and let $V$ be an indecomposable finitely generated non-projective Gorenstein-projective left $\A$-module whose stable endomorphism ring is isomorphic to $\k$. In this article, we prove that the universal deformation rings $R(\A,V)$ and $R(\A,\Omega_\A V)$ are isomorphic, where $\Omega_\A V$ denotes the first syzygy of $V$ as a left $\A$-module. We also prove the following result. Assume that $\A$ is also Gorenstein and that $\Gamma$ is another Gorenstein $\k$-algebra such that there exists $\ell \geq 0$ and a pair of bimodules $({_\Gamma}X_\A, {_\A}Y_\Gamma)$ that induces a singular equivalence of Morita type with level $\ell$ (as introduced by Z. Wang) between $\A$ and $\Gamma$. Then the left $\Gamma$-module $X\otimes_\A V$ is also Gorenstein-projective with stable endomorphism ring isomorphic to $\k$, and the universal deformation ring $R(\Gamma, X\otimes_\A V)$ is isomorphic to $R(\A, V)$.
\end{abstract}
\maketitle

\section{Introduction}\label{int}

Throughout this article, we assume that $\k$ is a fixed field of arbitrary characteristic. We denote by $\widehat{\Ca}$ the category of all complete local commutative Noetherian $\k$-algebras with residue field $\k$. In particular, the morphisms in $\widehat{\Ca}$ are continuous $\k$-algebra homomorphisms that induce the identity map on $\k$.  Let $\A$ be a fixed finite dimensional $\k$-algebra, and let $R$ be a fixed but arbitrary object in $\widehat{\Ca}$.  We denote by $R\A$ the tensor product of $\k$-algebras $R\otimes_\k\A$, and denote by $(R\A)^e$ the enveloping $R$-algebra of $R\A$, i.e. $(R\A)^e=R\A\otimes_R(R\A)^\textup{op}$, where $(R\A)^\textup{op}$ denotes the opposite $R$-algebra of $R\A$. In particular, all $R\A$-$R\A$-bimodules coincide with all left $(R\A)^e$-modules. Note that if $R$ is an Artinian object in $\widehat{\Ca}$, then $R\A$ is also Artinian (both on the left and the right sides). We denote by $R\A$-mod the abelian category of finitely generated left $R\A$-modules and by $R\A$-\underline{mod} its stable category.  In this article, we assume all our modules to be finitely generated. Let $M$ be a left $R\A$-module. We denote by $\End_{R\A}(M)$ (resp. by $\SEnd_{R\A}(M)$) the endomorphism ring (resp. the stable endomorphism ring) of $M$. If $R$ is an Artinian object in $\widehat{\Ca}$, then we denote by $\Omega_{R\A} M$ the first syzygy of $M$, i.e. $\Omega_{R\A} M$ is the kernel of a projective cover $P\to M$ of $M$ over $R\A$, which is unique up to isomorphism. In particular, if $N$ is a left $(R\A)^e$-module, we denote by $\Omega_{(R\A)^e}N$ the syzygy of $N$ as a $(R\A)^e$-module. Recall that $\A$ is said to be a {\it Gorenstein} $\k$-algebra provided that $\A$ has finite injective dimension as a left and right $\A$-module (see \cite{auslander2}). In particular, algebras of finite global dimension as well as self-injective algebras are Gorenstein.

Let $V$ be a fixed left $\A$-module. In \cite[Prop. 2.1]{blehervelez}, F. M. Bleher and the author proved that $V$ has a well-defined versal deformation ring $R(\A,V)$ in $\widehat{\Ca}$, which is universal provided that $\End_\A(V)$ is isomorphic to $\k$. Moreover, they also proved that versal deformation rings are preserved under Morita equivalences (see \cite[Prop. 2.5]{blehervelez}).  Following \cite{enochs0,enochs}, we say that $V$ is {\it Gorenstein-projective} provided that there exists an acyclic complex of projective left $\A$-modules 
\begin{equation*}
P^\bullet: \cdots\to P^{-2}\xrightarrow{f^{-2}} P^{-1}\xrightarrow{f^{-1}} P^0\xrightarrow{f^0}P^1\xrightarrow{f^1}P^2\to\cdots
\end{equation*}  
such that $\Hom_\A(P^\bullet, \A)$ is also acyclic and $V=\mathrm{coker}\,f^0$. In particular, every projective left $\A$-module is Gorenstein-projective. Following \cite{auslander4} and \cite[\S 2]{avramov2},  $V$ is said to have {\it Gorenstein dimension zero} or that it is {\it totally reflexive} provided that  $V$ is reflexive (i.e. $V$ and $\Hom_\A(\Hom_\A(V,\A),\A)$ are isomorphic as left $\A$-modules), and that $\Ext_\A^i(V,\A)=0=\Ext_\A^i(\Hom_\A(V,\A),\A)$ for all $i>0$. It is well-known that finitely generated Gorenstein-projective modules coincide with those that are totally reflexive (see e.g. \cite[\S 2.4]{avramov2}). Following \cite{buchweitz}, $V$ is said to be {\it (maximal) Cohen-Macaulay} provided that $\Ext_\A^i(V,\A)=0$, for all $i>0$. It follows by \cite[Prop. 4.1]{auslander3} that if $\A$ is a Gorenstein $\k$-algebra, then $V$ is a Gorenstein-projective left $\A$-module if and only if  $V$ is (maximal) Cohen-Macaulay.  It is important to mention that  there are examples of finite dimensional algebras $\A$ and left $\A$-modules $V$ that satisfy that $\Ext_\A^i(V,\A)=0$ for all $i>0$, but $V$ is not a Gorenstein-projective $\A$-module (see e.g. \cite[Example A.3]{hoshino-koga} and \cite{ringel-zhang,ringel-zhang2}). We denote by $\A$-Gproj the category of Gorenstein-projective left $\A$-modules, and by $\A\textup{-\underline{Gproj}}$ its stable category.  It is well-known that $\A$-Gproj is a Frobenius category in the sense of \cite[Chap. I, \S 2.1]{happel} and consequently,  $\A\textup{-\underline{Gproj}}$ is a triangulated category (in the sense of \cite{verdier}). Moreover, if $V$ is non-projective Gorenstein-projective, then for all $i\geq 0$, the $i$-th syzygy $\Omega_\A^iV$ is also non-projective Gorenstein-projective, and $\Omega_\A$ induces an autoequivalence $\Omega_\A:\A\textup{-\underline{Gproj}}\to \A\textup{-\underline{Gproj}}$ (see  \cite[Chap. I, \S2.2]{happel}). 

In \cite[Thm. 1.2]{bekkert-giraldo-velez}, it was proved that if $V$ is a Gorenstein-projective left $\A$-module with $\SEnd_\A(V)=\k$, then the versal deformation ring $R(\A,V)$ is universal, which generalizes \cite[Thm. 2.6 (ii) ]{blehervelez}. Moreover,  it was also proved that versal deformation rings of Gorenstein-projective modules are preserved under {\it singular equivalences of Morita type} between Gorenstein $\k$-algebras, which generalizes  \cite[Prop. 3.2.6]{blehervelez2}. These singular equivalences of Morita type were introduced by X. W. Chen and L. G. Sun in \cite{chensun} and then further studied by G. Zhou and A. Zimmermann in \cite{zhouzimm} as a way of generalizing the concept of stable equivalences of Morita type as introduced by M. Brou\'e in \cite{broue}.  


In order to state the main result of this article, we first need to recall the following definition due to Z. Wang (see \cite[Def. 2.1]{wang}) that generalizes the concept of singular equivalences of Morita type. 

\begin{definition}\label{defi:3.2}
Let $\A$ and $\Gamma$ be finite dimensional $\k$-algebras, and let $X$ be a $\Gamma$-$\A$-bimodule and $Y$ a $\A$-$\Gamma$-bimodule. We say that the pair $({_\Gamma}X_\A,{_\A}Y_\Gamma)$ induces a {\it singular equivalence of Morita type with level} $\ell\geq 0$ between $\A$ and $\Gamma$ and say that $\A$ and $\Gamma$ are {\it singularly equivalent of Morita type with level $\ell$} if the following conditions are satisfied:
\begin{enumerate}
\item $X$ is projective as a left $\Gamma$-module and as a right $\A$-module;
\item $Y$ is projective as a left $\A$-module and as a right $\Gamma$-module; 
\item $X\otimes_\A Y\cong \Omega_{\Gamma^e}^\ell \Gamma$ in $\Gamma^e$-\underline{mod};
\item $Y\otimes_\Gamma X\cong \Omega_{\A^e}^\ell \A$ in $\A^e$-\underline{mod}.  
\end{enumerate}
\end{definition}

The goal of this article is to prove the following result.

\begin{theorem}\label{thm1}
Let $\A$ be a finite dimensional $\k$-algebra and let $V$ be an indecomposable non-projective Gorenstein-projective left $\A$-module with $\SEnd_{\A}(V)=\k$. 
\begin{enumerate}
\item The stable endomorphism ring $\SEnd_\A(\Omega_\A V)$ is isomorphic to $\k$, and the universal deformation ring $R(\A, \Omega_\A V)$ is isomorphic to  $R(\A,V)$ in $\widehat{\Ca}$.
\item  Assume that $\A$ is Gorenstein and let $\Gamma$ be another finite dimensional Gorenstein $\k$-algebra such that there exists $\ell \geq 0$ and a pair of bimodules $({_\Gamma}X_\A,{_\A}Y_{\Gamma})$ that induces a singular equivalence of Morita type with level $\ell$ as in Definition \ref{defi:3.2}. Then $X\otimes_\A V$ is a non-projective Gorenstein-projective left $\Gamma$-module such that $\SEnd_{\Gamma}(X\otimes_\A V)=\k$ and the universal deformation ring $R(\Gamma, X\otimes_\A V)$ is isomorphic to $R(\A, V)$ in $\widehat{\Ca}$.
\end{enumerate}  

\end{theorem}

Note that Theorem \ref{thm1} (i) generalizes \cite[Thm. 2.6 (iv)]{blehervelez} and answers affirmatively a question raised in \cite[Rem. 5.5]{bekkert-giraldo-velez}. Moreover, Theorem \ref{thm1} (ii) gives a version of \cite[Thm. 1.2 (iii)]{bekkert-giraldo-velez} for singular equivalences of Morita type with level between Gorenstein $\k$-algebras. 

\begin{remark}
Let $\A$ be a finite dimensional $\k$-algebra.
\begin{enumerate}
\item It is important to mention that in the proof of Theorem \ref{thm1} (i), it is not enough to only use that $\Omega_\A$ is a self-equivalence of $\A\textup{-\underline{Gproj}}$ (as mentioned above), for we also need the following result from \cite[Prop. 3.8]{auslander4}: $V$ is totally reflexive if and only if $\Ext_\A^i(V,\A)=0=\Ext_{\A}^i(\mathrm{Tr}_{\A}V,\A)$ for all $i>0$, where $\mathrm{Tr}_{\A}V$ is the transpose of $V$ (see e.g. \cite[\S IV.1]{auslander}).   
\item In Example \ref{exam1} we show that Theorem \ref{thm1} (ii) fails provided that the conditions that $\A$ is Gorenstein and that $V$ is a non-projective Gorenstein projective left $\A$-module with $\SEnd_\A(V)\cong \k$ are not both satisfied. 
\end{enumerate}    
\end{remark}

The remainder of this article is organized as follows. In \S \ref{sec21}, we review the precise definition of lifts, deformations, universal deformations and universal deformations rings from \cite{blehervelez}. We also discuss some properties of singular equivalences of Morita type with level and syzygies of modules. In \S \ref{sec3} we prove Theorem \ref{thm1}. Finally, in \S \ref{sec4}, we provide some immediate applications of Theorem \ref{thm1} to Morita and triangular matrix $\k$-algebras as well as to singular equivalences induced by homological epimorphisms (in the sense of \cite{geigle}) and $2$-recollements of triangulated categories (in the sense of \cite{qin2}).

\section{Preliminaries}\label{sec2}
Throughout this section we keep the notation introduced in \S \ref{int}. 

\subsection{Lifts, deformations, and (uni)versal deformation rings}\label{sec21}

Let $V$ be a left $\A$-module and let $R$ be a fixed but arbitrary object in $\widehat{\Ca}$. A {\it lift} $(M,\phi)$ 
of $V$ over $R$ is a finitely generated left $R\A$-module $M$ 
that is free over $R$ 
together with an isomorphism of $\A$-modules $\phi:\k\otimes_RM\to V$. Two lifts $(M,\phi)$ and $(M',\phi')$ over $R$ are {\it isomorphic} 
if there exists an $R\A$-module 
isomorphism $f:M\to M'$ such that $\phi'\circ (\mathrm{id}_\k\otimes_R f)=\phi$.
If $(M,\phi)$ is a lift of $V$ over $R$, we  denote by $[M,\phi]$ its isomorphism class and say that $[M,\phi]$ is a {\it deformation} of $V$ 
over $R$. We denote by $\Def_\A(V,R)$ the 
set of all deformations of $V$ over $R$. The {\it deformation functor} corresponding to $V$ is the 
covariant functor $\widehat{\Fun}_V:\widehat{\Ca}\to \Sets$ defined as follows: for all objects $R$ in $\widehat{\Ca}$, define $\widehat{\Fun}_V(R)=\Def_
\A(V,R)$, and for all morphisms $\theta:R\to 
R'$ in $\widehat{\Ca}$, 
let $\widehat{\Fun}_V(\theta):\Def_\A(V,R)\to \Def_\A(V,R')$ be defined as $\widehat{\Fun}_V(\theta)([M,\phi])=[R'\otimes_{R,\theta}M,\phi_\theta]$, 
where $\phi_\theta: \k\otimes_{R'}
(R'\otimes_{R,\theta}M)\to V$ is the composition of $\A$-module isomorphisms 
\[\k\otimes_{R'}(R'\otimes_{R,\theta}M)\cong \k\otimes_RM\xrightarrow{\phi} V.\] 


Suppose there exists an object $R(\A,V)$ in $\widehat{\Ca}$ and a deformation $[U(\A,V), \phi_{U(\A,V)}]$ of $V$ over $R(\A,V)$ with the 
following property. For all objects $R$ in $\widehat{\Ca}$ and for all deformations $[M,\phi]$ of $V$ over $R$, there exists a morphism $\psi_{R(\A,V),R,[M,\phi]}:R(\A,V)\to R$ 
in $\widehat{\Ca}$ such that 
\[\widehat{\Fun}_V(\psi_{R(\A,V),R,[M,\phi]})[U(\A,V), \phi_{U(\A,V)}]=[M,\phi],\]
and moreover, $\psi_{R(\A,V),R,[M,\phi]}$ is unique if $R$ is the ring of dual numbers $\k[\epsilon]$ with $\epsilon^2=0$.  Then $R(\A,V)$ and $
[U(\A,V),\phi_{U(\A,V)}]$ are called the {\it versal deformation ring} and {\it versal deformation} of $V$, respectively. If the morphism $
\psi_{R(\A,V),R,[M,\phi]}$ is unique for all $R\in\Ob(\widehat{\Ca})$ and deformations $[M,\phi]$ of $V$ over $R$, then $R(\A,V)$ and $[U(\A,V),\phi_{U(\A,V)}]$ are 
called the {\it universal deformation ring} and the {\it universal deformation} of $V$, respectively.  In other words, the universal deformation 
ring $R(\A,V)$ represents the deformation functor $\widehat{\Fun}_V$ in the sense that $\widehat{\Fun}_V$ is naturally isomorphic to the $\Hom$ 
functor $\Hom_{\widehat{\Ca}}(R(\A,V),-)$. 

We denote by $\Fun_V$  the restriction of $\widehat{\Fun}_V$ to the full subcategory of Artinian objects in $\widehat{\Ca}$. Following \cite[\S 2.6]{sch}, we call the set $\Fun_V(\k[\epsilon])$ the tangent space of $\Fun_V$, which has a structure of a $\k$-vector space by \cite[Lemma 2.10]{sch}. 
It was proved in \cite[Prop. 2.1]{blehervelez} that $\Fun_V$ satisfies the Schlessinger's criteria \cite[Thm. 2.11]{sch}, that there exists an isomorphism of $\k$-vector spaces 
\begin{equation}\label{hoch}
\Fun_V(\k[\epsilon])\to \Ext_\A^1(V,V),
\end{equation}
and that $\widehat{\Fun}_V$ is continuous in the sense of \cite[\S 14]{mazur}, i.e. for all objects $R$ in $\widehat{\Ca}$, we have  
\begin{equation}\label{cont}
\widehat{\Fun}_V(R)=\invlim_n \Fun_V(R/\m_R^n), 
\end{equation}
where $\m_R$ denotes the unique maximal ideal of $R$. Consequently, $V$ has always a well-defined versal deformation ring $R(\A,V)$ which is also universal provided that $\End_\A(V)$ is isomorphic to $\k$. It was also proved in \cite[Prop. 2.5]{blehervelez} that versal deformation rings are invariant under Morita equivalences between finite dimensional $\k$-algebras.

\begin{remark}\label{univ0}
\begin{enumerate}
\item It follows from the isomorphism of $\k$-vector spaces (\ref{hoch}) that if $\dim_\k \Ext_\A^1(V,V)=r$, then the versal deformation ring $R(\A,V)$ is isomorphic to a quotient algebra of the power series ring  $\k[\![t_1,\ldots,t_r]\!]$ and $r$ is minimal with respect to this property. In particular, if $V$ is a left $\A$-module such that $\Ext_\A^1(V,V)=0$, then $R(\A,V)$ is universal and isomorphic to $\k$ (see \cite[Remark 2.1]{bleher15} for more details).

\item Because of the continuity of the deformation functor as in (\ref{cont}), most of the arguments concerning $\widehat{\Fun}_V$ can be carried out for $\Fun_V$, and thus we are able to restrict ourselves to discuss liftings of $\A$-modules over Artinian objects in $\widehat{\Ca}$. 

\item Let $R$ be an Artinian ring in $\widehat{\Ca}$, let $\iota_R:\k\to R$ be the unique morphism in $\widehat{\Ca}$ endowing $R$ with a $\k$-algebra structure, and let $\pi_R:R\to \k$ be the natural projection in $\widehat{\Ca}$. Then $\pi_R\circ \iota_R =\mathrm{id}_\k$. 
\begin{enumerate}
\item For all projective left (resp. right) $\A$-modules $P$, we let $P_R=R\otimes_{\k,\iota_R}P=R\otimes_\k P$. Then $P_R$ is a projective left (resp. right) $R\A$-module cover of $P$, and $(P_R, \pi_{P,R})$ is a lift of $P$ over $R$, where $\pi_{P,R}$ is the natural isomorphism $\k\otimes_{R, \pi_R}P_R\to P$. 
\item Let $\alpha: P(V)\to V$ be a projective left $\A$-module cover of $V$ (which is unique up to isomorphism), and let $\Omega_\A V=\ker\alpha$. Then we obtain a short exact  sequence of left $\A$-modules 
\begin{equation}\label{projcover}
0\to \Omega_\A V\xrightarrow{\beta}P(V)\xrightarrow{\alpha} V\to 0.
\end{equation}

Let $(M,\phi)$ be a lift of $V$ over $R$. Since $P_R(V)=R\otimes_{\k, \iota_R}P(V)$ is a projective left  $R\A$-module cover of $P(V)$ by (i), and  since $\alpha$ is an essential epimorphism, there exists an epimorphism of $R\A$-modules $\alpha_R: P_R(V)\to M$ such that $\phi\circ (\mathrm{id}_\k\otimes\alpha_R)= \alpha\circ \pi_{P(V),R}$. Moreover, it follows by \cite[Claim 1]{bleher15} that $\alpha_R: P_R(V)\to M$ is a projective left $R\A$-module cover of $M$. Let  $\Omega_{R\A}M:=\ker \alpha_R$. Note that since $M$ and $P_R(V)$ are both free over $R$, then $\Omega_{R\A}M$ is also free over $R$, and that there exists an isomorphism of left $\A$-modules $\Omega_{R\A}(\phi):\k\otimes_R\Omega_{R\A}M\to \Omega_\A V$ such that $\pi_{P(V),R}\circ (\mathrm{id}_\k\otimes\beta_R)=\beta\circ \Omega_{R\A}(\phi)$, where $\beta:\Omega_\A V\to P(V)$ and $\beta_R:\Omega_{R\A}M\to P_R(V)$ are the natural inclusions. In particular, $(\Omega_{R\A}M, \Omega_{R\A}(\phi))$ is a lift of $\Omega_\A V$ over $R$.
\end{enumerate}
\end{enumerate}  
\end{remark}

\begin{remark}\label{rem2.1}
Let $\A$ and $V$ be as above, let $R$ be an object in $\hat{\Ca}$ and let $(M,\phi)$ be a lift of $V$ over $R$. Then the isomorphism class $[M]$ of $M$ as an $R\A$-module is called a {\it weak deformation} of $V$ over $R$ (see e.g. \cite[\S 5.2]{keller} and \cite[Remark 2.4]{blehervelez}). We can also define the weak deformation functor $\hat{\Fun}_V^w: \hat{\Ca}\to\Sets$ which sends an object $R$ in $\hat{\Ca}$ to the set of
weak deformations of $V$ over $R$  and a morphism  $\alpha:R \to R'$  in $\hat{\Ca}$ to the map  $\hat{\Fun}_V^w :\hat{\Fun}_V^w (R)\to \hat{\Fun}_V^w(R')$, which is defined by $\hat{\Fun}_V^w(\alpha)([M]) = [R'\otimes_{R,\alpha}M]$.
In general, a weak deformation of $V$ over $R$ identifies more lifts than a deformation of $V$ over $R$ that respects the isomorphism $\phi$ of a representative $(M,\phi)$. 
\end{remark}

\begin{remark}\label{rem2.2}
Assume that $V$ is an indecomposable Gorenstein-projective left $\A$-module with $\SEnd_\A(V)=\k$. 
\begin{enumerate}
\item  It follows by \cite[Thm. 1.2 (i)]{bekkert-giraldo-velez} that the deformation functor $\widehat{\Fun}_V$ is naturally isomorphic to the weak deformation functor $\widehat{\Fun}_V^w$ as in Remark \ref{rem2.1}. This implies that a deformation $[M,\phi]$ of $V$ over $R$ in $\widehat{\Ca}$ does not depend on the particular choice of the $\A$-module isomorphism $\phi$. More precisely, if $f:M\to M'$ is an $R\A$-module isomorphism with $(M',\phi')$ a lift of $V$ over $R$, then there exists an $R\A$-module isomorphism $\bar{f}:M\to M'$ such that $\phi'\circ (\mathrm{id}_\k\otimes _R\bar{f})=\phi$, i.e., $[M,\phi]=[M',\phi']$ in $\widehat{\Fun}_V(R)=\Def_\A(V,R)$. 
\item As already noted in \S\ref{int}, it follows by \cite[Thm. 1.2 (ii)]{bekkert-giraldo-velez} that the versal deformation ring is $R(\A,V)$ is universal. Moreover, if $P$ is a projective left $\A$-module, then the versal deformation ring $R(\A,V\oplus P)$ is universal and isomorphic to $R(\A,V)$. This result follows from the fact that for all Artinian objects $R$ in $\widehat{\Ca}$, there is a bijection of set of deformations
\begin{align}
\tau_{V\oplus P,R}: \Fun_{V}(R) \to \Fun_{V\oplus P}(R)
\end{align}
which for all lifts $(M,\phi)$ of $V$ over $R$, $\tau_{P,R}([M,\phi])=[M\oplus P_R, \phi\oplus \pi_{P,R}]$, where $(P_R,\pi_{P,R})$  is as in Remark \ref{univ0} (iii.a).
\end{enumerate}
\end{remark}



\subsection{Some results involving singular equivalences of Morita type with level and syzygies}\label{section4}
Recall that $\A$ denotes a finite dimensional $\k$-algebra and $V$ is a finitely generated left $\A$-module. 

\begin{remark}
Assume that $\A$ has a minimal projective resolution as a $\A$-$\A$-bimodule (or equivalently, as a left $\A^e$-module) given by
\begin{equation}\label{projcoverA}
\cdots \to P_i\to P_{i-1}\to \cdots \to P_2\to P_1\to P_0\to \A\to 0.
\end{equation}
Let $i\geq 1$ be fixed.  It follows  that there exists a short exact sequence of $\A$-$\A$-bimodules 
\begin{equation}\label{projcoverAe}
0\to \Omega^i_{\A^e}\A\xrightarrow{\iota_i} P_i\xrightarrow{\pi_i} \Omega_{\A^e}^{i-1}\A\to 0.
\end{equation}
\noindent
Tensoring (\ref{projcoverAe}) with $V$ over $\A$ yields an exact sequence of left $\A$-modules 

\begin{equation*}\label{projcoverAeV}
\Omega_{\A^e}^i\A\otimes_\A V\xrightarrow{\iota_i\otimes\mathrm{id}_V} P_i\otimes_\A V\xrightarrow{\pi_i\otimes\mathrm{id}_V} \Omega^{i-1}_{\A^e}\A\otimes_\A V\to 0.
\end{equation*}
Since $\iota_i: \Omega_{\A^e}^i\A\to P_i$ is a section in $\A^e$-mod, it follows that 
\begin{equation}\label{mono}
\iota_i\otimes\mathrm{id}_V: \Omega_{\A^e}^i\A\otimes_\A V\to P_i\otimes_\A V
\end{equation}
is also a section, and thus a monomorphism in $\A$-mod. Thus we obtain a short exact sequence of left $\A$-modules 
\begin{equation*}\label{projcoverAeV0}
0\to \Omega^i_{\A^e}\A\otimes_\A V\xrightarrow{\iota_i\otimes\mathrm{id}_V} P_i\otimes_\A V\xrightarrow{\pi_i\otimes\mathrm{id}_V} \Omega_{\A^e}^{i-1}\A\otimes_\A V\to 0.
\end{equation*}
Thus by \cite[Prop. IV. 8.1 (v)]{skow3}, it follows that  there is an isomorphism of left $\A$-modules 
\begin{equation}\label{sumproj}
\Omega_{\A^e}^i\A\otimes_\A V\cong \Omega_\A^i V\oplus T'_i,
\end{equation} 
where $T'_i$ is a finitely generated projective left $\A$-module. 
\end{remark}
\begin{remark}\label{rem2.5}
Let $\Gamma$ be another finite dimensional $\k$-algebra, and assume that $\A$ and $\Gamma$ are both Gorenstein. Assume that there exists $\ell\geq 0$ and a pair of bimodules $({_\Gamma}X_\A,{_\A}Y_\Gamma)$ that  induce a singular equivalence of Morita type with level $\ell$ between $\A$ and $\Gamma$ as in Definition \ref{defi:3.2}. 
\begin{enumerate}
\item It follows from \cite[Lemma. 3.6]{skart} that the functors 
\begin{align*}
X\otimes_\A-:\A\textup{-mod} \to \Gamma \textup{-mod}&& \text{ and } && Y\otimes_\Gamma-:\Gamma\textup{-mod} \to \A\textup{-mod} 
\end{align*}
send finitely generated Gorenstein-projective left modules to finitely generated Gorenstein-projective left modules. 
\item By \cite[Prop. 2.3]{zhouzimm} and \cite[Prop. 3.7]{skart} it follows that 
\begin{align*}
X\otimes_\A-:\A\textup{-\underline{Gproj}}\to \Gamma\textup{-\underline{Gproj}}&& \text{ and } && Y\otimes_\Gamma-:
\Gamma\textup{-\underline{Gproj}}\to \A\textup{-\underline{Gproj}} 
\end{align*}
are equivalences of triangulated categories that are quasi-inverses of each other.

\item There exist projective bimodules ${_\Gamma} Q_\Gamma$, ${_\Gamma} Q'_\Gamma$, ${_\A}P_\A$, and ${_\A}P'_\A$ such that.

\begin{enumerate}
\item $X\otimes_\A Y\oplus Q'\cong \Omega_{\Gamma^e}^\ell\Gamma\oplus Q$ as $\Gamma$-$\Gamma$-bimodules;
\item $Y\otimes_\Gamma X\oplus P'\cong \Omega_{\A^e}^\ell\A\oplus P$ as $\A$-$\A$-bimodules.
\end{enumerate}
\noindent
Assume that $V$ is as in Remark \ref{rem2.2}. Then by tensoring at both sides of (iii.b) with $V$ over $\A$ and by using (\ref{sumproj}), we obtain an isomorphism of left $\A$-modules
\begin{equation*}
Y\otimes_\Gamma X\otimes_\A V\oplus (P'\otimes_\A V)\cong \Omega_\A^\ell V\oplus T'_\ell\oplus (P\otimes_\A V),
\end{equation*}
where $P'\otimes_\A V$ and $P\otimes_\A V$ are projective left $\A$-modules. This implies that  $\Omega_\A^\ell V$ and $Y\otimes_\Gamma X\otimes_\A V$ are isomorphic  indecomposable objects in $\A$-\underline{Gproj}. In particular, $\Omega_\A^\ell V$ does not have projective direct summands, and thus by the Krull-Schmidt-Azumaya Theorem and by (i), we obtain that there exists a finitely generated projective left $\A$-module $T_\ell$ such that there is an isomorphism of left $\A$-modules
\begin{equation}\label{eqn2.9}
Y\otimes_\Gamma X\otimes_\A V\cong \Omega_\A^\ell V\oplus T_\ell. 
\end{equation} 

Similarly, if $W$ is an indecomposable Gorenstein-projective left $\Gamma$-module, it follows that there exists a projective left $\Gamma$-module $S_\ell$ such that there exists an isomorphism of left $\Gamma$-modules
\begin{equation}\label{eqn2.10}
X\otimes_\A Y\otimes_\Gamma W\cong \Omega_\Gamma^\ell W\oplus S_\ell. 
\end{equation} 

\end{enumerate}
\end{remark}

\section{Proof of Theorem \ref{thm1}}\label{sec3}

Assume throughout this section that $\A$ is a fixed but arbitrary finite dimensional $\k$-algebra and that $V$ is an indecomposable non-projective Gorenstein-projective left $\A$-module with $\SEnd_\A(V)=\k$. 

\subsection{Proof of Theorem \ref{thm1} (i)}

We first need to recall the following definition from \cite[Def. 1.2]{sch}.

\begin{definition}\label{defi3.7}
Let $\theta:R\to R_0$ be a morphism of Artinian objects in $\widehat{\Ca}$. We say that $\theta$ is a {\it small extension} if the kernel of $\theta$ is a non-zero principal ideal $tR$ that is annihilated by the unique maximal ideal $\m_R$ of $R$.
\end{definition}

\begin{lemma}\label{lemma2.4}
Assume that  $R$ is a fixed but arbitrary Artinian object in $\widehat{\Ca}$. 
\begin{enumerate}
\item Let $M$ be a finitely generated left $R\A$-module such that $\k\otimes_RM\cong V$ as left $\A$-modules. Then:
\begin{enumerate}
\item $M$ is a (maximal) Cohen-Macaulay left $R\A$-module, i.e., $\Ext_{R\A}^i(M,R\A)=0$ for all $i>0$;
\item $M$ is reflexive as a left $R\A$-module, i.e. $M$ and $\Hom_{R\A}(\Hom_{R\A}(M,R\A),R\A)$ are isomorphic as left $R\A$-modules.
\end{enumerate}
\item Let $U$ be a finitely generated left $R\A$-module which is free over $R$ such that $\k\otimes_RU\cong \Omega_\A V$ as left $\A$-modules. Then there exists a short exact sequence of left $R\A$-modules that are free over $R$
\begin{equation}\label{omegalift}
0\to U\xrightarrow{\beta_R}P_R(V)\xrightarrow{\alpha_R} L\to 0,
\end{equation}
such that the induced short exact sequence of left $\A$-modules
\begin{equation*}
0\to \k\otimes_RU\xrightarrow{\mathrm{id}_\k\otimes \beta_R}\k\otimes_RP_R(V)\xrightarrow{\mathrm{id}_\k\otimes \alpha_R} \k\otimes_RL\to 0,
\end{equation*}
is isomorphic to (\ref{projcover}).
\end{enumerate}
\end{lemma}

\begin{proof}
(i.a). Let $R_0$ be an Artinian object in $\widehat{\Ca}$ such that $\theta: R\to R_0$ is a small extension as in Definition \ref{defi3.7}, and assume that for all finitely generated left $R_0\A$-modules such that $\k\otimes_{R_0}M_0\cong V$ as left $\A$-modules, we have $\Ext_{R_0\A}^i(M_0,R_0\A)=0$ for all $i>0$. Consider the short exact sequence of $R$-modules 
\begin{equation}\label{smallext}
0\to tR\to R\to R_0\to 0. 
\end{equation}
Tensoring (\ref{smallext}) with $\A$ over $\k$ yields a sequence of left $R\A$-modules
\begin{equation}\label{RAseq}
0\to tR\otimes_\k\A\to R\A\to R_0\A\to 0. 
\end{equation} 
Applying $\Hom_{R\A}(M,-)$ to (\ref{RAseq}) yields a long exact sequence of left $R\A$-modules 
\begin{equation*}
\cdots \to \Ext_{R\A}^i(M, tR\otimes_\k\A)\to \Ext_{R\A}^i(M,R\A)\to \Ext_{R\A}^i(M,R_0\A)\to \Ext_{R\A}^{i+1}(M, tR\otimes_\k\A)\to \cdots
\end{equation*}
Since $tR\cong \k$, it follows that $tR\otimes_\k\A\cong \A$. On the other hand, by using \cite[Prop. 1.8.31]{zimmermann}, the fact that $\k\otimes_RM\cong V$ is (maximal) Cohen-Macaulay, and by using the induction hypothesis, we obtain that for all $i>0$, $\Ext_{R\A}^i(M,tR\otimes_\k\A)\cong \Ext^i_\A(\k\otimes_RM,\A)=0$ and $\Ext^i_{R\A}(M,R_0\A)\cong \Ext_{R_0\A}^i(R_0\otimes_{R,\theta}M_0,R_0\A)=0$. Therefore, $\Ext_{R\A}^i(M,R\A) =0 $ for all $i>0$.

Note that in the proof of (i.a), we only use that $\Ext_\A^i(V,\A)=0$ for all $i>0$. 

(i.b). Let $P_1\xrightarrow{d_1}P_0\xrightarrow{d_0} V\to 0$ be a minimal projective presentation of $V$ as a left $\A$-module. Then by using Remark \ref{univ0} (iii.b) repeatedly, we obtain a minimal projective presentation of $M$ as a left $R\A$-module
\begin{equation}\label{minR}
(P_1)_R\xrightarrow{(d_1)_R}(P_0)_R\xrightarrow{(d_0)_R} M\to 0,
\end{equation}
where $\k\otimes_R(P_i)_R\cong P_i$ for $i=0,1$. Applying $\Hom_{R\A}(-,R\A)=(-)^{\ast_{R\A}}$ to (\ref{minR}) yields the exact sequence of right $R\A$-modules
\begin{equation*}
(P_0)_R^{\ast_{R\A}}\xrightarrow{(d_1)_R^{\ast_{R\A}}}(P_1)_R^{\ast_{R\A}}\to \mathrm{Tr}_{R\A}M\to 0,
\end{equation*}
where $\mathrm{Tr}_{R\A}M:=\mathrm{coker}\,(d_1)_R^{\ast_{R\A}}$ is the transpose of $M$ (see e.g. \cite[\S IV.1]{auslander}). By \cite[Prop. 6.3]{auslander6} (see also \cite[Intro.]{auslander4}), there exists an exact sequence of left $R\A$-modules
\begin{equation}\label{ausexact}
0\to \Ext_{R\A}^1(\mathrm{Tr}_{R\A}M,R\A)\to M\to M^{\ast_{R\A}\ast_{R\A}}\to \Ext_{R\A}^2(\mathrm{Tr}_{R\A}M,R\A)\to 0.
\end{equation}
On the other hand, since $\k\otimes_R (P_i)_R^{\ast_{R\A}}\cong \Hom_\A(P_i,\A)$ for $i=0,1$, it follows that $\k\otimes_R \mathrm{Tr}_{R\A}M\cong \mathrm{Tr}_\A V$. Since $V$ is also totally reflexive, it follows by \cite[Prop. 3.8]{auslander4} that $\Ext_\A^i(\mathrm{Tr}_\A V,\A)=0$ for all $i>0$. Thus by using the dual arguments in the proof of (i.a) for the right $R\A$-module $\mathrm{Tr}_{R\A}M$, we obtain $\Ext_{R\A}^i(\mathrm{Tr}_{R\A}M,R\A)=0$ for all $i>0$, which together with (\ref{ausexact}) implies that $M\to M^{\ast_{R\A}\ast_{R\A}}$ is an isomorphism of $R\A$-modules. 

(ii). In the following, we adapt some of the arguments in the proof of \cite[Prop. 2.4]{bleher2} to our situation. As in the proof of (i.a), let $R_0$ be an Artinian object in $\widehat{\Ca}$ such that $\theta: R\to R_0$ is a small extension as in Definition \ref{defi3.7}, and assume that for all finitely generated left $R_0\A$-modules $U_0$ that are free over $R_0$ and satisfy that $\k\otimes_{R_0}U_0\cong \Omega_\A V$ as left $\A$-modules, there exists a monomorphism of left $R_0\A$-modules $\beta_{R_0}:U_0\to P_{R_0}(V)$ such that $\mathrm{id}_\k\otimes \beta_{R_0}= \beta$, where $\beta$ is the monomorphism of left $\A$-modules as in (\ref{projcover}).  
Tensoring (\ref{smallext}) with $P_R(V)$ yields an exact sequence of left $R\A$-modules 
\begin{equation*}
0\to tR\otimes_RP_R(V)\to P_R(V)\to P_{R_0}(V)\to 0.
\end{equation*}
Since $tR\cong \k$, it follows that $tR\otimes_RP_R(V)\cong P(V)$, and together with the fact that $\k\otimes_R U\cong  \Omega_\A V$ is a Gorenstein-projective left $\A$-module, we obtain that $\Ext_{R\A}^1(U,tR\otimes_RP_R(V))\cong \Ext_{\A}^1(\k\otimes_R U,P(V))=0$. On the other hand, $\Hom_{R\A}(U,R_0\A)\cong \Hom_{R_0\A}(U_0,R_0\A)$, where $U_0=R_0\otimes_{R,\theta}U$ is a finitely generated $R_0\A$-module that is free over $R_0$ and which satisfies $\k\otimes_{R_0} U_0\cong \Omega_{\A}V$. Thus we obtain a short exact sequence
\begin{equation*}
0\to \Hom_{\A}(\k\otimes_RU,P(V))\to \Hom_{R\A}(U,P_R(V))\to \Hom_{R_0\A}(U_0,P_{R_0}(V))\to 0.
\end{equation*}
Let $\beta_{R_0}: U_0\to P_{R_0}(V)$ be a monomorphism of left $R_0\A$-modules such that $\mathrm{id}_\k\otimes \beta_{R_0}= \beta$. Then there exists $\beta_R: U\to P_R(V)$ such that $\mathrm{id}_{R_0}\otimes \beta_R= \beta_{R_0}$. Therefore, $\mathrm{id}_\k\otimes\beta_R=\beta$ and since $\beta$ is a monomorphism of left $\A$-modules, it follows by Nakayama's Lemma that $\beta_R: U\to P_R(V)$ is also a monomorphism of left $R\A$-modules. By letting $L=\mathrm{coker}\, \beta_R$, we obtain a short exact sequence of left $R\A$-modules as in (\ref{omegalift}). Note that since $U$ and $P_R(V)$ are both free over $R$, it follows that $L$ is also free over $R$. Moroever, since $\k\otimes_R U\cong \Omega_{\A}V$ and $\k\otimes_RP_R(V)\cong V$, it follows that $\k\otimes_R L\cong V$ as left $\A$-modules. This finishes the proof of Lemma \ref{lemma2.4}.

\end{proof}

Let $R$ be a fixed Artinian object in $\widehat{\Ca}$ and let $(M,\phi)$ be a lift of $V$ over $R$. Then by Remark \ref{univ0} (iii.b), we obtain that $(\Omega_{R\A}M, \Omega_{R\A}\phi)$ is a lift of $\Omega_\A V$ over $R$. Thus we obtain a map between set of deformations
\begin{equation*}
\tau_{\Omega_\A V,R}: \Fun_V(R)\to \Fun_{\Omega_\A V}(R)
\end{equation*}
defined as $\tau_{\Omega_\A V,R}([M,\phi])=[\Omega_{R\A}M,\Omega_{R\A}\phi]$ for all $[M,\phi]\in \Fun_V(R)$.  

Assume that $[M,\phi]=[M',\phi']$ in $\Fun_V(R)$. Then there is an isomorphism of left $R\A$-modules $f:M\to M'$ such that $\phi'\circ (\mathrm{id}_\k\otimes f)=\phi$. In particular, we obtain an isomorphism of left $R\A$-modules $\Omega_{R\A}f: \Omega_{R\A}M\to \Omega_{R\A}M'$. By Remark \ref{rem2.2} (i), it follows that $[\Omega_{R\A}M,\Omega_{R\A}\phi]= [\Omega_{R\A}M',\Omega_{R\A}\phi']$ in  $\Fun_{\Omega_\A V}(R)$, which proves that $\tau_{\Omega_\A V,R}$ is well-defined. Next let $(U,\varphi)$ be a lift of $\Omega_\A V$ over $R$. Then by Lemma \ref{lemma2.4} (ii), there exists a lift $(L,\psi)$ of $V$ over $R$ such that $[\Omega_{R\A}L,\Omega_{R\A}\psi]=[U,\varphi]$. This proves that $\tau_{\Omega_\A V,R}$ is surjective. In order to prove that $\tau_{\Omega_\A V,R}$ is injective, we adjust (as before) some of the arguments in the proof of \cite[Prop. 2.4]{bleher2} to our situation. Namely,  assume  that $[M_1,\phi]$ and $[M_2,\phi_2]$ are lifts of $V$ over $R$ such that $[\Omega_{R\A}M_1,\Omega_{R\A}\phi_1]=[\Omega_{R\A}M_2,\Omega_{R\A}\phi_2]$ in $\Fun_{\Omega_\A V}(R)$. Then for $i=1,2$, there exists a short exact sequence of left $R\A$-modules
\begin{equation*}
0\to \Omega_{R\A}M_i\to P_R(V)\to M_i\to 0.
\end{equation*}
Since by Lemma \ref{lemma2.4} (i.a) we have that $\Ext_{R\A}^1(M_i,R\A)=0$ for $i=1,2$, we obtain a short exact sequence of right $R\A$-modules
\begin{equation*}
0\to M_i^{\ast_{R\A}}\to (P_R(V))^{\ast_{R\A}}\to (\Omega_{R\A}M_i)^{\ast_{R\A}}\to 0,
\end{equation*}
where $(P_R(V))^{\ast_{R\A}}$ is a projective right $R\A$-module. Since $R\A$ is an Artinian $R$-algebra on both sides, it follows by Schanuel's Lemma and the Krull-Schmidt-Azumaya Theorem that $M_1^{\ast_{R\A}}\cong M_2^{\ast_{R\A}}$ as right $R\A$-modules and thus by Lemma \ref{lemma2.4} (i.b) it follows that $M_1\cong M_1^{\ast_{R\A}\ast_{R\A}}\cong M_2^{\ast_{R\A}\ast_{R\A}}\cong M_2$ as left $R\A$-modules. Therefore by Remark \ref{rem2.2} (i) we obtain that $[M_1,\phi_1]=[M_2,\phi_2]$ in $\Fun_V(R)$. This proves that $\tau_{\Omega_\A V,R}$ is injective. Finally, let $\theta: R\to R'$ be a morphism between Artinian rings in $\widehat{\Ca}$. Then it is straightforward to prove that $\Omega_{R'}(R'\otimes_{R,\theta}M)\cong R'\otimes_{R,\theta}\Omega_{R\A}M$ as left $R'\A$-modules. Thus by using Remark \ref{rem2.2} (i) again, we obtain that $\tau_{\Omega_\A V,R}$ is natural with respect to morphisms between Artinian objects in $\widehat{\Ca}$. 

The continuity of the deformation functor (see Remark \ref{univ0} (ii)) implies that for all objects $R$ in $\widehat{\Ca}$, there is a bijection between sets of deformations
\begin{equation*}
\widehat{\tau}_{\Omega_\A V,R}: \widehat{\Fun}_V(R)\to \widehat{\Fun}_{\Omega_\A V}(R),
\end{equation*}
which is natural with respect of morphisms between objects in $\widehat{\Ca}$. This implies that the universal deformation rings $R(\A,V)$ and $R(\A,\Omega_\A V)$ are isomorphic in $\widehat{\Ca}$. This finishes the proof of Theorem \ref{thm1} (i). 

The following remark will be useful in the proof of Theorem \ref{thm1} (ii).

\begin{remark}\label{rem3.3}
Let $R$ be an Artinian object in $\widehat{\Ca}$. By using the arguments above, it follows that for all $i\geq 1$, there is a bijection of set of deformations 
\begin{equation}\label{eqn3.6}
\tau_{\Omega_\A^i V,R}: \Fun_V(R)\to \Fun_{\Omega_\A^i V}(R),
\end{equation}
which is natural with respect to morphisms between Artinian objects in $\widehat{\Ca}$. 
\end{remark}

\subsection{Proof of Theorem \ref{thm1} (ii)}

Assume that $R$ is a fixed but arbitrary Artinian object in $\widehat{\Ca}$.

\begin{lemma}\label{lemma3.4}
Let $(M,\phi)$ be a lift of $V$ over $R$. Then for all $i\geq 0$, there is an isomorphism of left $R\A$-modules
\begin{equation}\label{sumprojR}
\Omega_{(R\A)^e}^i(R\A)\otimes_{R\A}M\cong \Omega_{R\A}^iM\oplus (T'_i)_R
\end{equation} 
where $\Omega_{R\A}^iM$ is as in Remark \ref{univ0} (iii.b), $T'_i$ is as in (\ref{sumproj}), and $(T'_i)_R$ is the projective left $R\A$-module as in Remark \ref{univ0} (iii.a).
\end{lemma}
\begin{proof}
Consider the minimal projective resolution of $\A$ as a $\A$-$\A$-bimodule in (\ref{projcoverA}). Tensoring with $R$ over $\k$ yields a minimal projective resolution of $R\A$ as an $R\A$-$R\A$-bimodule
\begin{equation}\label{projcoverRA}
\cdots \to (P_i)_R\to (P_{i-1})_R\to \cdots \to (P_2)_R\to (P_1)_R\to (P_0)_R\to R\A\to 0,
\end{equation}
where for all $i\geq 0$, $(P_i)_R$ is as in Remark \ref{univ0} (iii.a). Let $i\geq 0$ be fixed but arbitrary, and consider the short exact sequence of $R\A$-$R\A$-bimodules 
\begin{equation}\label{projcoverRAe}
0\to \Omega_{(R\A)^e}^i(R\A)\xrightarrow{(\iota_i)_R}(P_i)_R\xrightarrow{(\pi_i)_R} \Omega_{(R\A)^e}^{i-1}(R\A)\to 0.
\end{equation}
\noindent
Tensoring (\ref{projcoverRAe}) with $M$ over $R\A$ yields an exact sequence 

\begin{equation}\label{projcoverRAeM}
\Omega_{(R\A)^e}^i(R\A)\otimes_{R\A}M\xrightarrow{(\iota_i)_R\otimes\mathrm{id}_M}(P_i)_R\otimes_{R\A}M\xrightarrow{(\pi_i)_R\otimes\mathrm{id}_M} \Omega_{(R\A)^e}^{i-1}(R\A)\otimes_{R\A}M\to 0,
\end{equation}
where $(P_i)_R\otimes_{R\A}M$ is a projective left $R\A$-module. 
Note that tensoring the morphism 
\begin{equation*}
(\iota_i)_R\otimes\mathrm{id}_M: \Omega_{(R\A)^e}^i(R\A)\otimes_{R\A}M\xrightarrow{(\iota_i)_R\otimes\mathrm{id}_M}(P_i)_R\otimes_{R\A}M
\end{equation*}
with $\mathrm{id}_\k$ over $R$, induces the monomorphism (\ref{mono}), and thus by Nakayama's Lemma, we obtain that $(\iota_i)_R\otimes\mathrm{id}_M$ is a monomorphism. Thus we obtain a short exact sequence of left $R\A$-modules
\begin{equation*}
0\to \Omega_{(R\A)^e}^i(R\A)\otimes_{R\A}M\xrightarrow{(\iota_i)_R\otimes\mathrm{id}_M}(P_i)_R\otimes_{R\A}M\xrightarrow{(\pi_i)_R\otimes\mathrm{id}_M} \Omega_{(R\A)^e}^{i-1}(R\A)\otimes_{R\A}M\to 0.
\end{equation*}
\noindent
Note that $\Omega_{(R\A)^e}^{i-1}(R\A)\otimes_{R\A}M$ induces a lift of $\Omega_{\A^e}^{i-1}\A\otimes_\A V$ over $R$.
This together with the discussion in Remark \ref{univ0} (iii.b) implies that $\Omega_{(R\A)^e}^i(R\A)\otimes_{R\A}M\cong \Omega_{R\A}(\Omega_{(R\A)^e}^{i-1}(R\A)\otimes_{R\A}M)\oplus Z_i$ for some projective left $R\A$-module $Z_i$.  Note that since (\ref{sumprojR}) is trivially true for when $i=0$, we can assume by induction that for all $0\leq j < i$, $\Omega_{(R\A)^e}^{j}(R\A)\otimes_{R\A}M\cong \Omega_{R\A}^{j}M\oplus (T'_{j})_R$, where $T'_{j}$ is as in (\ref{sumproj}). Thus 
\begin{align}\label{projcoverRAi}
\Omega_{(R\A)^e}^i(R\A)\otimes_{R\A}M&\cong \Omega_{R\A}(\Omega_{R\A}^{i-1}M\oplus (T'_{i-1})_R)\oplus Z_i\cong\Omega_{R\A}^iM\oplus Z_i. 
\end{align}
Note that the isomorphism (\ref{projcoverRAi}) of left $R\A$-modules implies that $Z_i$ induces a lift of $T'_i$ over $R$, where $T'_i$ is as in (\ref{sumproj}). Thus $Z_i$ is isomorphic to $(T'_i)_R$. This finishes the proof of Lemma  \ref{lemma3.4}. 
\end{proof}
 
In the following, assume further that  $\A$, $\Gamma$, $\ell$, ${_\Gamma}X_\A$, ${_\A}Y_\Gamma$, $P$, $P'$, $Q$ , and $Q'$ are all as in Remark \ref{rem2.5}. 
\begin{remark}\label{rem3.5}
It follows that $X_R=R\otimes_\k X$ is projective as a left $R\Gamma$-module and as a right $R\A$-module, and $Y_R=R\otimes_\k Y$ is projective 
as a left $R\A$-module and as a right $R\Gamma$-module, and both are free over $R$. Note also that $X_R\otimes_{R\A}Y_R\cong R\otimes_\k(X\otimes_\A Y)$ as $R\Gamma$-$R\Gamma$-bimodules and $Y_R\otimes_{R\Gamma} X_R\cong R\otimes_\k (Y\otimes_\A X)$ as $R\A$-$R\A$-bimodules. Moreover, we also have that 
\begin{align*}
Y_R\otimes_{R\Gamma} X_R\oplus P'_R&\cong \Omega_{R\A^e}^\ell(R\A) \oplus P_R&&\text{ as $R\A$-$R\A$-bimodules, and}\\
X_R\otimes_{R\A} Y_R\oplus Q'_R&\cong \Omega_{R\A^e}^\ell(R\Gamma)\oplus Q_R&&\text{ as $R\Gamma$-$R\Gamma$-bimodules},
\end{align*}
where $P_R=R\otimes_\k P$ and $P'_R=R\otimes_\k P'$ (resp. $Q_R=R\otimes_\k Q$ and $Q'_R=R\otimes_\k Q'$) are projective  $R\A$-$R\A$-bimodules (resp. $R\Gamma$-$R\Gamma$-bimodules).
\end{remark}

Let $(M,\phi)$ be a lift of $V$ over $R$. Since $V$ is assumed to be indecomposable, we obtain by using Remark \ref{rem3.5} together with Lemma \ref{lemma3.4} that there is an isomorphism of $R\A$-$R\A$-bimodules
\begin{equation}\label{eqn3.11}
Y_R\otimes_{R\Gamma}X_R\otimes_{R\A} M\cong \Omega_{R\A}^\ell M\oplus (T_\ell)_R, 
\end{equation}
where $T_\ell$ is as in (\ref{eqn2.9}). Note also that $X_R\otimes_{R\A} M$ is free over $R$, and that there exists an isomorphism of left $\Gamma$-modules $\phi_{X_R\otimes_{R\A} M}: \k\otimes_R(X_R\otimes_{R\A} M)\to X\otimes_\A V$.  Thus we can define a morphism between sets of deformations 
\begin{equation}
\tau_{X\otimes_\A V, R}: \Fun_V(R)\to\Fun_{X\otimes_\A V}(R) 
\end{equation}
as follows. For all deformations $[M,\phi]$ of $V$ over $R$, let $\tau_{X\otimes_\A V, R}([M,\phi])=[X_R\otimes_{R\A}M,\phi_{X_R\otimes_{R\A}M}]$. Let $(M,\phi)$ and $(M',\phi')$ be lifts of $V$ over $R$ such that $[M,\phi]=[M',\phi']$ in $\Fun_V(R)$. It follows that there is an isomorphism of left $R\Gamma$-modules $g: X_R\otimes_{R\A}M\to X_R\otimes_{R\A}M'$. Note that by  Remark \ref{rem2.5} (ii), we obtain that $X\otimes_\A V$ is an indecomposable non-projective Gorenstein-projective left $\Gamma$-module with $\SEnd_{\Gamma}(X\otimes_\A V)=\k$. Thus by Remark \ref{rem2.2} (i), we obtain that $[X_R\otimes_{R\A} M, \phi_{X_R\otimes_{R\A} M}]=[X_R\otimes_{R\A} M', \phi_{X_R\otimes_{R\A} M'}]$ in $\Fun_{X\otimes_\A V}(R)$. This proves that $\tau_{X\otimes_\A V,R}$ is well-defined. Next assume that $[N,\varphi]$ is a deformation of $X\otimes_\A V$ over $R$. If we let $L=Y_R\otimes_{R\Gamma} N$, then $L$ is free over $R$ and there is a composition of isomorphisms between left $\A$-modules which we denote by $\phi_L$ and which is given as follows:
\begin{align*}
\k\otimes_RL= \k\otimes_R(Y_R\otimes_{R\Gamma} N)&\cong (\k\otimes_R Y_R)\otimes_\Gamma (\k\otimes_R N)\\ 
&\cong Y\otimes_\Gamma (X\otimes_\A V)\\
&\cong \Omega_{\A}^\ell V\oplus T_\ell,
\end{align*}
where the last isomorphism follows from (\ref{eqn2.9}). In particular, $(L,\phi_L)$ is a lift of $\Omega_\A^\ell V\oplus T_\ell$ over $R$. Thus
by Remark \ref{rem2.2} (ii), there exists a lift $(L',\phi_{L'})$ of $\Omega_\A^\ell V$ over $R$ such that $L'\oplus (T_\ell)_R\cong L$ as left $R\A$-modules. On the other hand, by Remark \ref{rem3.3}, there exists a lift $(L'',\phi_{L''})$ of $V$ over $R$ such that $\Omega_{R\A}^\ell L''\cong L'$ as left $R\A$-modules. Therefore, $Y_R\otimes_{R\Gamma} N\cong \Omega_{R\A}^\ell L''\oplus (T_\ell)_R$ as left $R\A$-modules.  Note that we also have that $Y_R\otimes_{R\Gamma}X_R\otimes_{R\A} L''\cong \Omega_{R\A}^\ell L'\oplus (T_\ell)_R$ as left $R\A$-modules. Thus we obtain an isomorphism of left $R\A$-modules $Y_R\otimes_{R\Gamma} N\cong Y_R\otimes_{R\Gamma}X_R\otimes_{R\A} L''$, which induces a composition of isomorphisms between left $R\Gamma$-modules as follows:
\begin{align*}
\Omega_{\Gamma}^\ell N\oplus (S_\ell)_R &\cong X_R\otimes_{R\A}Y_R\otimes_{R\Gamma} N\\
&\cong X_R\otimes_{R\A}Y_R\otimes_{R\Gamma}X_R\otimes_{R\A} L''\\
& \cong \Omega_{R\Gamma}^\ell(X_R\otimes_{R\A} L'')\oplus (S_\ell)_R,
\end{align*}
where $S_\ell$ is as in (\ref{eqn2.10}). Thus $\Omega_{R\Gamma}^\ell N\cong \Omega_{R\Gamma}^\ell (X_R\otimes_{R\A} L'')$ as left $R\Gamma$-modules. By Remark \ref{rem3.3} (applied to $X\otimes_\A V$) together with Remark \ref{rem2.2} (i), we obtain that $N\cong X_R\otimes_{R\A}L''$ as left $R\Gamma$-modules and that $[N,\varphi]= [X_R\otimes_{R\A}L'',\phi_{X_R\otimes_{R\A}L''}]$ in $\Fun_{X\otimes_\A V}(R)$. This proves that $\tau_{X\otimes_\A V, R}$ is surjective. Next assume that $[M,\phi]$ and $[M',\phi']$ are deformations of $V$ over $R$ such that $[X_R\otimes_{R\A}M,\phi_{X_R\otimes_{R\A}M}]=[X_R\otimes_{R\A}M',\phi_{X_R\otimes_{R\A}M'}]$. In particular, assume that there exists an isomorphism of left $R\Gamma$-modules $f: X_R\otimes_{R\A}M\to X_R\otimes_{R\A}M'$. Thus we obtain an isomorphism of left $R\A$-modules $\mathrm{id}_{Y_R}\otimes f: Y_R\otimes_{R\Gamma}X_R\otimes_{R\A} M\to Y_R\otimes_{R\Gamma}X_R\otimes_{R\A} M'$. By (\ref{eqn3.11}), we obtain that there exists an isomorphism of left $R\A$-modules $\Omega_{R\A}^\ell f: \Omega_{R\A}^\ell M\to \Omega_{R\A}^\ell M'$ such that $\mathrm{id}_\k \otimes {\Omega_{R\A}^\ell f} =\mathrm{id}_{\Omega_\A^\ell V}$. This together with Remark \ref{rem2.2} (i) implies that $[\Omega_{R\A}^\ell M,\Omega_{R\A}^\ell \phi]= [\Omega_{R\A}^\ell M',\Omega_{R\A}^\ell\phi']$ in $\Fun_{\Omega_\A^\ell V}(R)$, which together with Remark \ref{rem3.3} implies that $[M,\phi] = [M',\phi']$ in $\Fun_V(R)$. This proves that $\tau_{X\otimes_\A V, R}$ is injective. Next assume that $\theta: R\to R'$ is a morphism of Artinian objects in $\widehat{\Ca}$. Let $(M,\phi)$ a lift of $V$ over $R$. Then there is a composition of left $R'\A$-modules as follows:
\begin{align*}
R'\otimes_{R,\theta} (X_R\otimes_{R\A} M)\cong (R'\otimes_{R,\theta}X_R)\otimes_{R'\otimes_{R,\theta}R\A}(R'\otimes_{R,\theta}M)\cong X_{R'}\otimes_{R'\A}M',
\end{align*} 
where $X_{R'}=R'\otimes_\k X$ and $M'= R'\otimes_{R,\theta} M$. This proves that $\tau_{X\otimes_\A V, R}$ is natural with respect of morphism between Artinian objects in $\widehat{\Ca}$. 

The continuity of the deformation functor (see Remark \ref{univ0} (ii)) implies that for all objects $R$ in $\widehat{\Ca}$, there is a bijection between sets of deformations
\begin{equation*}
\widehat{\tau}_{X\otimes_\A V, R}: \widehat{\Fun}_V(R)\to \widehat{\Fun}_{X\otimes_\A V}(R),
\end{equation*}
which is natural with respect of morphisms between objects in $\widehat{\Ca}$. Consequently, we obtain that the universal deformation rings $R(\A,V)$ and $R(\A,X\otimes_\A V)$ are isomorphic in $\widehat{\Ca}$. This finishes the proof of Theorem \ref{thm1} (ii).

In the following, we provide an example (due to \O. Skarts{\ae}terhagen) of two finite dimensional $\k$-algebras $\A$ and $\Gamma$  which are singularly equivalent of Morita type with level (as in Definition \ref{defi:3.2}) and which show that Theorem \ref{thm1} (ii) can fail if both of the conditions for $\A$ and $V$ are not satisfied. 
\begin{example}\label{exam1}
Let $\A= \k Q/\langle \rho\rangle$ and $\Gamma = \k Q'/\langle \sigma \rangle$ be the monomial $\k$-algebras given by the following quivers with relations: 
\begin{align*}
Q: &\xymatrix{\underset{1}{\bullet}\ar@(ul,dl)_{\alpha}\ar[r]^{\beta} & \underset{2}{\bullet}},  & \rho=\{\alpha^2,\beta\alpha\},&\\
Q': &\xymatrix{\underset{3}{\bullet}\ar@(ul,dl)_{\gamma}}, & \sigma=\{\gamma^2\}.&
\end{align*}
It follows by \cite[Example 7.5]{skart} that $\A$ and $\Gamma$ are singularly equivalent of Morita type with level $\ell = 1$ as in Definition \ref{defi:3.2}. Moreover, $\A$ is a non-Gorenstein $\k$-algebra with radical square zero. It follows by \cite{chenxw3} that $\A$ is CM-free, i.e., every Gorenstein-projective left $\A$-module is projective. Thus if $V$ is a Gorenstein-projective left $\A$-module,  then $\SEnd_{\A}(V)=0$ and  $R(\A,V)\cong \k$ (by Remark \ref{univ0}). On the other hand, note that $\Gamma$ is a self-injective Nakayama $\k$-algebra with Loewy length equal to $2$. If $S_3$ denotes the simple $\Gamma$-module corresponding to the vertex of $Q'$, then $\SEnd_{\Gamma}(S_3) = \k$, and $\Ext_{\Gamma}^1(S_3,S_3)\not=0$. Thus by  \cite[Thm. 1.3 (ii)]{bleher15} we have that $R(\Gamma,S_3)\cong \k[\![t]\!]/(t^2)$. This shows that Theorem \ref{thm1} (ii) fails provided that both of the conditions $\A$ being Gorenstein and that $V$ being an non-projective indecomposable Gorenstein-projective with $\SEnd_{\A}(V)=\k$ are omitted in the hypothesis. 
\end{example}

\section{Applications}\label{sec4}
In the following, we discuss some immediate applications of the main results in this article. 

\subsection{Morita and triangular matrix algebras}
Let $\A$ and $\Gamma$ be finite dimensional $\k$-algebras, $B$ a $\A$-$\Gamma$-bimodule, and $C$ a $\Gamma$-$\A$-bimodule. We define the Morita $\k$-algebra
\begin{equation}\label{moritaalg}
\Sigma =\begin{pmatrix} \A & B\\C&\Gamma\end{pmatrix},
\end{equation} 
where the addition of elements of $\Sigma$ is componentwise and the multiplication is given by
\begin{equation*}
\begin{pmatrix} \lambda&b\\c&\gamma\end{pmatrix}\cdot \begin{pmatrix} \lambda'&b'\\c'&\gamma'\end{pmatrix}=\begin{pmatrix} \lambda\lambda' &\lambda b'+b\gamma'\\c\lambda'+\gamma c'&\gamma\gamma' \end{pmatrix},
\end{equation*} 
for all $\lambda,\lambda'\in \A$, $\gamma, \gamma'\in \Gamma$, $b,b'\in B$ and $c,c'\in C$. We refer the reader to \cite[\S 2.1]{gaopsa} for a detailed description of the abelian category $\Sigma$-mod as well as for a description of $\Sigma$-Gproj. The following result follows from \cite[Cor. 4.10]{gaopsa}.

\begin{lemma}\label{lemma4.1}
Let $\Sigma$ be a Morita $\k$-algebra as in (\ref{moritaalg}) which is also Gorenstein. 
\begin{enumerate}
\item If $C$ is projective as a right $\A$-module and $B$ is projective as a left $\A$-module, then the $\k$-algebra $\A$ is Gorenstein. 
\item If $B$ is projective as a right $\Gamma$-module and $C$ is projective as a left $\Gamma$-module, then the $\k$-algebra $\Gamma$ is Gorenstein. 
\end{enumerate}
\end{lemma} 

The following result follows from Lemma \ref{lemma4.1} and \cite[Example 4.6]{dalezios}.

\begin{lemma}\label{lemma4.2}
Let $\Sigma$ be as in the hypothesis of Lemma \ref{lemma4.1}.
\begin{enumerate}
\item Under the situation of Lemma \ref{lemma4.1} (i), if $\Gamma$ has finite projective dimension as a $\Sigma$-$\Sigma$-bimodule, then there exist a pair of bimodules $({_\Sigma} X_\A, {_\A}Y_\Sigma)$ that induces a singular equivalence of Morita type with level $\ell$ between $\Sigma$ and $\A$ (as in Definition \ref{defi:3.2}), where $\ell$ is equal to the projective dimension of $\Gamma$ as a $\Gamma$-$\Gamma$-bimodule. 
\item Under the situation of Lemma \ref{lemma4.1} (ii), if $\A$ has finite projective dimension as a $\Sigma$-$\Sigma$-bimodule, then there exist a pair of bimodules $({_\Sigma} X'_\Gamma, {_\Gamma}Y'_\Sigma)$ that induces a singular equivalence of Morita type with level $\ell$ between $\Sigma$ and $\Gamma$ (as in Definition \ref{defi:3.2}), where $\ell$ is equal to the projective dimension of $\A$ as a $\A$-$\A$-bimodule. 
\end{enumerate}
\end{lemma} 
The following result is an immediate consequence of Lemmata \ref{lemma4.1} and \ref{lemma4.2} together with Theorem \ref{thm1} (ii).

\begin{corollary}
Let $\Sigma$ be as in Lemma \ref{lemma4.1}, and let $W$ be an indecomposable Gorentein-projective left $\Sigma$-module with $\SEnd_\Sigma(W)=\k$. 
\begin{enumerate}
\item If $\Sigma$ is under the situation of Lemma \ref{lemma4.2} (i), then $\A$ is also a Gorenstein $\k$-algebra,  $Y\otimes_\Sigma W$ is also an indecomposable Gorenstein-projective left $\A$-module with $\SEnd_{\A}(Y\otimes_\Sigma W)=\k$ and the universal deformation rings $R(\Sigma, W)$ and $R(\A, Y\otimes_\Sigma W)$ are isomorphic in $\widehat{\Ca}$.
\item If $\Sigma$ is under the situation of Lemma \ref{lemma4.2} (ii), then $\Gamma$ is also a Gorenstein $\k$-algebra,  $Y'\otimes_\Sigma W$ is also an indecomposable Gorenstein-projective left $\Gamma$-module with $\SEnd_{\Gamma}(Y'\otimes_\Sigma W)=\k$ and the universal deformation rings $R(\Sigma, W)$ and $R(\Gamma, Y'\otimes_\Sigma W)$ are isomorphic in $\widehat{\Ca}$.
\end{enumerate}
\end{corollary}

\begin{remark}\label{rem4.4}
Assume next that $\Sigma$ is as the triangular matrix $\k$-algebra 

\begin{equation}\label{triagalg}
\Sigma =\begin{pmatrix} \A & B\\0&\Gamma\end{pmatrix},
\end{equation} 
where $\A$ has infinite global dimension, $\Gamma$ has finite projective dimension as a $\Gamma$-$\Gamma$-bimodule equal to $m_1$, and that $B$ is a $\A$-$\Gamma$-bimodule with finite projective dimension equal to $m_2$. Then by the analogous results in \cite[Claim 3.1]{wang} for $\Sigma$, there exists a pair of bimodules $({_\Sigma}X_\A, {_\A}Y_{\Sigma})$ that induce a singular equivalence of Morita type with level $\ell+1$ between $\Sigma$ and $\A$ (as in Definition \ref{defi:3.2}), where $\ell =\min\{m_1,m_2\}$. Assume further that $\Sigma$ is Gorenstein. It follows by \cite[Thm. 2.2]{zhang} that $\A$ and $\Gamma$ are both Gorenstein. For a description of the objects in $\Sigma$-Gproj under this situation, we refer the reader to \cite[Thm. 1.4]{zhang}.
\end{remark}

The following result (which improves that in \cite[Thm. 1.3]{velez3}) follows immediately from Remark \ref{rem4.4} and Theorem \ref{thm1} (ii).

\begin{corollary}
Under the situation and notation in Remark \ref{rem4.4}, if $W$ is an indecomposable Gorenstein-projective left $\Sigma$-module with $\SEnd_{\Sigma}(W)=\k$, then $Y\otimes_\Sigma W$ is also an indecomposable Gorenstein-projective left $\A$-module with $\SEnd_\A(Y\otimes_\Sigma W)=\k$, and the universal deformation rings $R(\Sigma, W) $ and $R(\A, Y\otimes_{\Sigma} W)$ are isomorphic in $\widehat{\Ca}$. 
\end{corollary}

\subsection{Singular equivalences induced by homological epimorphisms}

\begin{remark}\label{rem4.6}
Let $\A$ be a finite dimensional  $\k$-algebra and let $J\subseteq \A$ be an ideal.  Following \cite{delapena}, we say that 
$J$ is a {\it homological ideal} provided that the canonical map $\A\to \A/J$ is a homological epimorphism (in the sense of \cite{geigle}). It follows from the main result in \cite{chenxw2} that in this situation, if we further have that $J$ has finite projective dimension as a $\Gamma$-$\Gamma$-bimodule, then there is a singular equivalence between $\A$ and $\A/J$. In particular, if $\A$ and $\A/J$ are both Gorenstein, then the triangulated categories $\A$-\underline{Gproj} and $\A/J$-\underline{Gproj} are equivalent.  Moreover, it follows from \cite[Thm. 3.6]{qin} that $\A$ and $\A/J$ are singularly equivalent of Morita type with level $\ell$ (as in Definition \ref{defi:3.2}). 
\end{remark}

The following result follows immediately from Remark \ref{rem4.6} and Theorem \ref{thm1} (ii).

\begin{corollary}\label{cor4.7}
Let $\A$ be a Gorenstein $\k$-algebra and let $J$ a homological ideal of $\A$ such that $\Gamma = \A/J$ is also Gorenstein. Let $({_\Gamma}X_\A, {_\A}Y_\Gamma)$ be a pair of bimodules that induces a singular equivalence of Morita type with level $\ell$ (as in Definition \ref{defi:3.2}). If $V$ is an indecomposable Gorenstein-projective left $\A$-module with $\SEnd_\A(V)=\k$, then $X\otimes_\A V$ is also an indecomposable Gorenstein-projective left $\Gamma$-module with $\SEnd_{\Gamma}(X\otimes_\A V)=\k$ and the universal deformation rings $R(\A,V)$ and $R(\Gamma, X\otimes_\A V)$ are isomorphic in $\widehat{\Ca}$. 
\end{corollary}

In the following, we discuss an example (due to X. W. Chen) of two finite dimensional $\k$-algebras that verify Corollary \ref{cor4.7}.

\begin{example}

Consider the basic $\k$-algebras $\A$ and $\Gamma$ as in Figures \ref{fig1} and \ref{fig2}, respectively. It follows from \cite[Exam. 3.5]{chenxw2} and Remark \ref{rem4.6} that $\A$ is a Gorenstein $\k$-algebra with injective dimension $2$ at both sides and that there exists a singular equivalence of Morita type with level $\ell$ (as in Definition \ref{defi:3.2}) between $\A$ and $\Gamma$.  
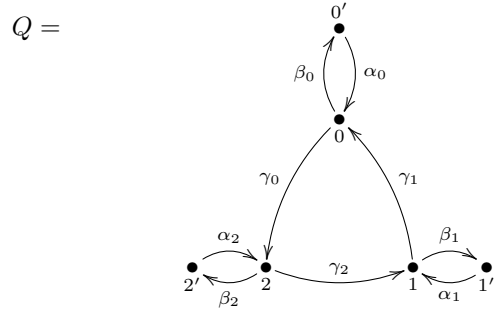
\begin{figure}[ht]
\begin{align*}
Q &=\hspace*{1.5cm}\xymatrix@1@=20pt{
&&\overset{0'}{\bullet}\ar@/^/[d]^{\alpha_0}&&\\
&&\underset{0}{\bullet}\ar@/^/[u]^{\beta_0}\ar@/_/[ddl]_{\gamma_0}&&\\\\
\underset{2'}{\bullet}\ar@/^/[r]^{\alpha_2}&\underset{2}{\bullet}\ar@/^/[l]^{\beta_2}\ar@/_/[rr]^{\gamma_2}&&\underset{1}{\bullet}\ar@/_/[uul]_{\gamma_1}\ar@/^/[r]^{\beta_1}&\underset{1'}{\bullet}\ar@/^/[l]^{\alpha_1}
}\\\\
\rho_0&= \{\beta_i\alpha_i, \gamma_i\alpha_i, \beta_{i-1}\gamma_i, \alpha_i\beta_i-(\gamma_{i+1}\gamma_{i+2}\gamma_i)^2 : i\in \Z/3\}.
\end{align*}
\caption{The basic $\k$-algebra $\A_0=\k Q/\langle \rho_0\rangle$.}\label{fig1}
\end{figure}

\begin{figure}[ht]
\begin{align*}
Z_3 &=\xymatrix@1@=20pt{
&&\underset{0}{\bullet}\ar@/_/[ddl]_{\gamma_0}&&\\\\
&\underset{2}{\bullet}\ar@/_/[rr]^{\gamma_2}&&\underset{1}{\bullet}\ar@/_/[uul]_{\gamma_1}&
}\\\\
\rho_1&= \{(\gamma_0,\gamma_1,\gamma_2)^{6}:i\in \Z/3\}.
\end{align*}
\caption{The basic $\k$-algebra $\Gamma=\k Z_3/\langle \rho_1\rangle $.}\label{fig2}
\end{figure}
\noindent
Note also that $\Gamma$ is a self-injective (thus Gorenstein) Nakayama $\k$-algebra, and that $\A$ is a special biserial algebra (in the sense of \cite{buri}) of finite representation type. Thus we can describe combinatorially  the indecomposable non-projective objects in $\A$-mod by using so-called strings for $\A$; the corresponding $\A$-modules are called string modules. The morphisms between these string $\A$-modules can be completely described by using the results in \cite{krause}.  Moreover, it follows from \cite[Prop. 3.1 (b)]{auslander3} that $\A\textup{-Gproj}=\Omega^2(\A\textup{-mod})$. Using the above arguments together with the description of the irreducible morphisms between string $\A$-modules in \cite{buri}, we can identify all the indecomposable non-projective Gorenstein-projective left $\A$-modules in the stable Auslander-Reiten quiver of $\A$. More precisely, we have that the indecomposable non-projective Gorenstein-projective left $\A$-modules are given as follows:
\begin{align*}
V_{i,0}&=M[\alpha_{i+1}^{-1}\gamma_{i+2}\gamma_i\gamma_{i+1}\gamma_{i+2}\gamma_i],&V_{i,1}&=M[\alpha_{i+1}^{-1}\gamma_{i+2}\gamma_i\gamma_{i+1}\gamma_{i+2}],\\
V_{i,2}&=M[\alpha_{i+1}^{-1}\gamma_{i+2}\gamma_i\gamma_{i+1}], &V_{i,3}&=M[\alpha_{i+1}^{-1}\gamma_{i+2}\gamma_i], \\
V_{i,4}&=M[\alpha_{i+1}^{-1}\gamma_{i+2}],
\end{align*}
where $i\in \Z/3$.  It is straightforward to check that for all $i\in \Z/3$, $\Omega V_{i,0} = V_{i+2, 4}$, $\Omega V_{i,1}= V_{i+1,3}$ and $\Omega V_{i,2} = V_{i,2}$.  Then it follows from \cite{krause} and \cite[Lemma 5.2]{skart} that for all $i\in \Z/3$ and for $j\in \{0,1,2,3,4\}$, $\SEnd_{\A_0}(V_{i,j})\cong \k$,  $\Ext_{\A_0}^1(V_{i,j}, V_{i,j})=0$ with $j\not=3$, and $\Ext_{\A}^1(V_{i,3}, V_{i,3})\cong \k$ as $\k$-vector spaces. Thus for $j\in \{0,1,2,3,4\}$, $R(\A, V_{i,j})$ is universal and isomorphic to $\k$ for $j\not=3$, and $R(\A,V_{i,3})$ is a quotient of $\k[\![t]\!]$. Let $i\in \Z/3$ be fixed and denote by $P_{i'}$ the incomposable projective left $\A$-modules corresponding to the verticex $i'$ of $Q$. Then there exists a non-splitting short sequence of left $\A$-modules
\begin{equation}\label{ses}
0\to V_{i,3}\xrightarrow{\iota}P_{i'}\oplus V_{i,0}\xrightarrow{\pi} V_{i,3}\to 0. 
\end{equation} 
\noindent
If we let $M= P_{i'}\oplus V_{i,0}$, then $M$ defines a non-trivial a left $\k[\![t]\!]/(t^2)\A$-module by letting $t$ act on $m\in M$ as $t\cdot m= (\iota \circ \pi)(m)$. Thus there exists a unique surjective $\k$-algebra homomorphism $\theta: R(\A, V_{i,3})\to \k[\![t]\!]/(t^2)$ in $\widehat{\Ca}$ corresponding to the deformation defined by $M$. Assume that $\theta$ is not an isomorphism. Thus there exists a surjective $\k$-algebra homomorphism $\theta': R(\A, V_{i,3})\to \k[\![t]\!]/(t^3)$ in $\widehat{\Ca}$ such that $\pi_{3,2}\circ \theta' = \theta$, where $\pi_{3,2}: \k[\![t]\!]/(t^3)\to \k[\![t]\!]/(t^2)$ is the natural projection. Let $M'$ be a left $\k[\![t]\!]/(t^3)\A$-module that defines a lift of $V_{i,3}$ over $\k[\![t]\!]/(t^3)$ corresponding to $\theta'$. Note that $M'/t^2M'\cong M$ and $t^2 M'\cong V_{i,3}$. Thus, we obtain a short exact sequence of $\k[\![t]\!]/(t^3)\A$-modules
\begin{equation}\label{splisseq}
0\to V_{i,3}\to M'\to M\to 0. 
\end{equation}
Note that since $\Ext_{\A}^1(M,V_{i,3})=\SHom_\A(\Omega_\A V_{i,0},V_{i,3})=0$, it follows that (\ref{splisseq}) splits as a sequence of left $\A$-modules. Hence $M'=V_{i,3}\oplus M$ as left $\A$-modules. Thus if $(v,m)\in M'$ with $v\in V_{i,3}$ and $m\in M$, it follows that the action of $t$ on $M'$ is given by $t\cdot (v,m) = (\sigma(m), t\cdot m)$, where $\sigma: M\to V_{i,3}$ is a surjective $\A$-module homomorphism. Then there exists $c\in \k^\ast$ such that $\sigma = c\pi$, where $\pi$ is as in (\ref{ses}), and thus the kernel of $\sigma$ is $tM$. This implies that $\sigma(tm)=0=t^2m$ for all $m\in M$, and consequently, $t^2(v,m)=(\sigma(tm),t^2m)=(0,0)$ for all $v\in V_{i,3}$ and $m\in M$, which contradicts that  $t^2M'\cong V_{i,3}$. Thus $\theta$ is a $\k$-algebra isomorphism and $R(\A, V_{i,3})\cong \k[\![t]\!]/(t^2)$.  Therefore, if $V$ is an indecomposable Gorenstein-projective left $\A$-module, then $\SEnd_{\A}(V)=\k$ and the universal deformation ring $R(\A, V)$ is isomorphic either to $\k$ or to $ \k[\![t]\!]/(t^2)$. On the other hand, by using \cite[Thm. 1.2]{bleher15} it is straightforward to see that if $V'$ is an indecomposable non-projective (thus a Gorenstein-projective) left $\Gamma$-module, then $\SEnd_{\Gamma}(V')=\k$ and the universal deformation ring $R(\Gamma, V')$ is isomorphic either to $\k$ or to $\k[\![t]\!]/(t^2)$. This verifies Corollary \ref{cor4.7}. 
\end{example}

\subsection{Singular equivalences induced by $2$-recollements}
To end this section, we provide a result involving universal deformation rings of Gorenstein-projective modules over Gorenstein algebras and $2$-recollements (as introduced in \cite[Def. 2]{qin2}.   

Let $\mathscr{T}'$, $\mathscr{T}$, and $\mathscr{T}''$ be triangulated categories. Following \cite{beillinson-bernstein-deligne}, a {\it recollement} of $\mathscr{T}$ relative to $\mathscr{T}'$ and $\mathscr{T}''$ is given by 
\begin{equation}
\xymatrix@1@=60pt{\mathscr{T}'\ar[r]^{i_\ast= i_!}&\mathscr{T}\ar@/^1.5pc/[l]^{i^!}\ar@/_1.5pc/[l]_{i^\ast}\ar[r]^{j^!= j^\ast}&\mathscr{T}''\ar@/^1.5pc/[l]^{j_\ast}\ar@/_1.5pc/[l]_{j_!}
}
\end{equation}
such that 
\begin{itemize}
\item[(R1)] $(i^\ast, i_\ast)$, $(i_!, i^!)$, $(j_!, j^!)$ and $(j^\ast, j_\ast)$ are adjoint pairs of triangulated functors;
\item[(R2)] $i_\ast, j_!$ and $j_\ast$ are full embeddings;
\item[(R3)] $ j^!i_\ast = 0$ (and thus we also have $i^!j_\ast = 0$ and $i^\ast j_! = 0$); 
\item[(R4)] for each object $X$ in $\mathscr{T}$, there are triangles
\begin{align*}
j_!j^!x\to X\to i_\ast i ^\ast X\to\\
i_!i^!x\to X\to j_\ast j ^\ast X\to
\end{align*}
\noindent
where the arrows to and from $X$ are the counits and the units of the adjoint pairs respectively. 
\end{itemize}

Following \cite{qin2}, a {\it $2$-recollement} of $\mathscr{T}$ relative to $\mathscr{T}'$ and $\mathscr{T}''$ is given by a diagram of functors of triangulated categories 
\begin{equation}\label{2recoll}
\xymatrix@1@=60pt{\mathscr{T}'\ar@/^0.5pc/[r]^{i_2}\ar@/_2pc/[r]_{i_4}&\mathscr{T}\ar@/^0.5pc/[l]^{i_3}\ar@/_2pc/[l]_{i_1}\ar@/^0.5pc/[r]^{j_2}&\mathscr{T}''\ar@/^0.5pc/[l]^{j_3}\ar@/_2pc/[l]_{j_1}\ar@/^2pc/[l]^{j_4}
}
\end{equation}
such that each of the two possible consecutive three layers form a recollement of $\mathscr{T}$ relative to $\mathscr{T}'$ and $\mathscr{T}''$. 

For all finite dimensional $\k$-algebras $\A$, we denote by $\mathcal{D}(\A)=\mathcal{D}(\A\textup{-Mod})$ the derived category of all modules over $\A$ from the left side, which is a triangulated category (see e.g. \cite[Chap. I]{hartshorne}).  

The following result is due to Y. Qin (see \cite[Cor. 3.3]{qin}).

\begin{lemma}\label{sing2recoll}
Let $\A$, $\Gamma$ and $\Sigma$ be finite dimensional $\k$-algebras such that $\Gamma$ has finite projective dimension as a $\Gamma$-$\Gamma$-bimodule. Assume that $\mathscr{T}=\mathcal{D}(\Sigma)$ admits a $2$-recollement relative to $\mathscr{T}'=\mathcal{D}(\Gamma)$ and $\mathscr{T}''=\mathcal{D}(\A)$ as in (\ref{2recoll}). Then $\Sigma$ and $\A$ are singularly equivalent of Morita type with level as in Definition \ref{defi:3.2}.
\end{lemma}

The following result follows immediately from Lemma \ref{sing2recoll} and Theorem \ref{thm1} (ii).

\begin{corollary}
Assume that $\A$, $\Gamma$ and $\Sigma$ are as in Lemma \ref{sing2recoll}, with  $\A$ and $\Sigma$ Gorenstein $\k$-algebras. If  $W$ is a Gorenstein-projective left $\Sigma$-module with $\SEnd_\Sigma(W)=\k$, then there exists a Gorenstein-projective left $\A$-module $V$ such that $\SEnd_\A(V)=\k$ and the universal deformation rings $R(\Sigma, W)$ and $R(\A,V)$ are isomorphic in $\widehat{\Ca}$. 
\end{corollary}

\section{Acknowledgments}  
The author wants to express his gratitude to the anonymous referee, who provided many suggestions and corrections that helped the readability and quality of this article.

\bibliographystyle{amsplain}
\bibliography{Morita_Type_with_Level-Rev}

\end{document}